\documentclass{amsart}
\usepackage[latin1]{inputenc}
\usepackage[english]{babel}
\usepackage[T1]{fontenc}
\usepackage{amsmath}
\usepackage{amssymb}
\usepackage{amsthm}
\usepackage{array,booktabs}
\usepackage{graphicx}
\usepackage{newtxtext}
\usepackage{quoting}
\usepackage{enumerate}
\usepackage{mathtools}
\usepackage{csquotes}
\usepackage{cite}
\usepackage{tikz}
\usepackage{tikz-cd}
\usetikzlibrary{matrix,arrows,decorations.pathmorphing}
\usepackage{wrapfig}
\renewcommand{\Re}{\operatorname{Re}}
\renewcommand{\Im}{\operatorname{Im}}
\def\R{\ensuremath\mathbb{R}}
\def\A{\ensuremath\mathbb{A}}

\def\Z{\ensuremath\mathbb{Z}}
\def\Q{\ensuremath\mathbb{Q}}

\def\H{\ensuremath\mathbb{H}}

\def\fatC{\ensuremath {\bf C}}
\usepackage{hyperref}
\usepackage[final]{pdfpages}
\usepackage{verbatim}
\newtheorem{thm}{Theorem}[section]

\newtheorem{cor}[thm]{Corollary}
\newtheorem{lemma}[thm]{Lemma}
\newtheorem{prop}[thm]{Proposition}
\theoremstyle{remark}
\newtheorem{remark}[thm]{Remark}

\def\eps{\ensuremath\varepsilon}

\def\Cl {\text{\rm Cl }}

\def\Hess{\text{\rm Hess}}

\def\PSL{\hbox{\rm PSL}}
\def\GL{\hbox{\rm GL}}
\def\modulo{\text{ \rm mod }}
\def\x{\ensuremath \boldsymbol{x}}
\def\X{\ensuremath \boldsymbol{X}}
\def\y{\ensuremath \boldsymbol{y}}
\def\z{\ensuremath \boldsymbol{z}}
\def\m{\ensuremath \boldsymbol{\mu}}
\numberwithin{equation}{section}
\begin{document}
\title[Hybrid subconvexity and uniform sup norm bounds of Eisenstein series]{Hybrid subconvexity for class group $L$-functions and uniform sup norm bounds of Eisenstein series}
\author{Asbjørn Christian Nordentoft}
\address{Department of Mathematical Sciences, University of Copenhagen, Universitetsparken 5, 2100 Copenhagen Ø, Denmark}
\email{\href{mailto:acnordentoft@outlook.com}{acnordentoft@outlook.com}}
\date{\today}
\begin{abstract}
In this paper we study hybrid subconvexity bounds for class group $L$-functions associated to quadratic extensions $K/\Q$ (real or imaginary). Our proof relies on relating the class group $L$-functions to Eisenstein series evaluated at Heegner points using formulas due to Hecke. The main technical contribution is the following uniform sup norm bound for Eisenstein series;
$$E(z,1/2+it)\ll_\eps y^{1/2} (|t|+1)^{1/3+\eps},\quad y\gg 1,$$ 
extending work of Blomer and Titchmarsh. Finally we propose a uniform version of the sup norm conjecture for Eisenstein series. 
\end{abstract}
\maketitle
\author
\section{Introduction}
This paper is concerned with the family of $L$-functions $L_K(s,\chi)$ associated to a character $\chi$ of the (wide) class group $\Cl (K)$ of a quadratic field extension $K/\Q$ (real or imaginary) of discriminant $D$. One of our results is a hybrid subconvexity bound in terms of the discriminant $D$ and the archimedian parameter $t$ where $s=1/2+it$ (both for individual class group $L$-functions and for the second moment of the entire family). We will do this by relating the subconvexity bound for class group $L$-functions to sup norm bounds of Eisenstein series via formulas due to Hecke. Our second main result is what we will call a {\it uniform sup norm bound} of Eisenstein series. 
\subsection{Class group $L$-functions}The study of analytic properties of the family of class group $L$-functions was initiated by Duke, Friedlander and Iwaniec in \cite{DuFrIw95} where they computed the second moment of class group $L$-functions in the limit $D\rightarrow -\infty$. Other notable works on the family of class group $L$-functions include \cite{Bl04}, \cite{DuFrIw02}, \cite{BlHaMi07} \cite{Te11}.\\
Our approach in the imaginary quadratic case is to use a classical formula of Hecke, which relates class group $L$-functions to Eisenstein series evaluated at Heegner points; 
\begin{align}
 \label{heegnersum} L_K(s, \chi)= \frac{2^{s+1} \zeta(2s) |D|^{-s/2}}{\omega_K } \sum_{\mathfrak{a}}\chi(\mathfrak{a}) E(z_{\mathfrak{a}},s),
 \end{align}
where the sum runs over a complete set of representatives for the class group of the imaginary quadratic field $K$ of discriminant $D$, $z_{\mathfrak{a}}\in \H$ is the associated Heegner point and $\omega_K \in \{2,4,6\}$. There is a real quadratic analogue also due to Hecke (see (\ref{heegnerintegral}) below). These formulas give a connection between subconvexity bounds and the so-called {\it sup norm problem} for Eisenstein series, which we will introduce shortly. 
\begin{remark}The connection between the sup norm problem and subconvexity estimates can be traced back to Sarnak \cite[(4.19)]{Sa95}. However this paper together with the recent work of Hu and Saha \cite{HuSaha19} seem to be the first time sup norm results have been used to obtain new subconvexity results. Hu and Saha apply sup norm bounds of automorphic forms on quaternion algebras (in the depth aspect) to obtain subconvexity estimates in the depth aspect for $L(1/2, f \otimes \theta_\chi)$, where $f$ is a quaternionic automorphic form and $\theta_\chi$ is an essentially fixed theta series.
\end{remark}
\begin{remark}The formula (\ref{heegnersum}) was also the starting point for Templier in \cite{Te11}, where it was combined with equidistribution of Heegner points to give an alternative computation (compared with \cite{DuFrIw95}) of the second moment of the family of class group $L$-functions as $D\rightarrow -\infty$. Similarly Michel and Venkatesh \cite{MichelVenk07} used an analogue of (\ref{heegnersum}) in the case of cusp forms due to Zhang \cite{Zh01}, \cite{Zh04} to deduce non-vanishing results for the central values of the corresponding Rankin-Selberg $L$-functions. The approach of Michel and Venkatesh was then applied by Dittmer, Proulx and Seybert in \cite{DiPrSe15} to deduce non-vanishing for class group $L$-functions as well (their method only shows non-vanishing for one class group character for each $K$, whereas Blomer in \cite{Bl04} achieved a much stronger result using mollification). \end{remark}

\subsection{The sup norm problem}Now let $\Gamma_0(1)=\PSL_2(\Z)$ and denote by $X_0(1):=\Gamma_0(1)\backslash \H$ the modular curve. The sup norm problem for $X_0(1)$ is concerned with bounds of the following form for some fixed $\theta>0$;
$$  \sup_{z\in C} |u_j(z)| \ll_{C} t_j^{\theta}, $$
where $u_j$ is a Maass form of level 1, $t_j$ is the spectral parameter and $C\subset \H$ is compact. The case $\theta=1/4+\eps$ is known as the {\it convexity bound} and is elementary to prove, but it is conjectured \cite[Conjecture 3.10]{Sa95} that any $\theta>0$ is admissible. Iwaniec and Sarnak in their seminal paper \cite{IwSa} were the first to go beyond the convexity bound by proving the bound $\ll_\eps t_j^{5/24+\eps}$.  \\ 
In this paper we will focus on the analogue for the continuous spectrum which is constituted by Eisenstein series. This means that we are concerned with bounds of the type
\begin{align} \label{eissup} \sup_{z\in C} |E(z,1/2+it)| \ll_{C} (|t|+1)^{\theta}, \end{align}
where $\theta>0$ is fixed and $C$ is compact. In this case the convexity bound is $\theta=1/2+\eps$, and again the sup norm conjecture predicts that any $\theta>0$ is admissible. Iwaniec and Sarnak's method also applies in this case and yields similarly the bound $\ll_\eps (|t|+1)^{5/12+\eps}$. In \cite{Yo18} Young used a slight modification of the Iwaniec--Sarnak method to prove the bound $\ll_\eps (|t|+1)^{3/8+\eps}$. In \cite{Bl18} Blomer improved this using exponential sum methods, building on earlier work of Titchmarsh \cite{Ti}, and proved the Weyl type bound $\ll_\eps (|t|+1)^{1/3+\eps}$. Finally the sup norm problem for Eisenstein series over general number fields has been dealt with in the work of Assing \cite{Assing19}.\\ 

Plugging Blomer's result into (\ref{heegnersum}) yields immediately a subconvexity bound for $L_K(s, \chi)$ in the $t$-aspect, which recovers a result of S\"{o}hne \cite{Soehne97} (the conductor of $L_K(s, \chi)$ is $|D|(|t|+1)^2 $, which means that the convexity bound is $\ll_\eps |D|^{1/4+\eps}(|t|+1)^{1/2+\eps} $). If one however wants a hybrid subconvexity estimate, one needs to control the $D$-dependence in (\ref{heegnersum}). This leads to what we will call the {\it uniform sup norm problem}, which are sup norm bounds with an explicit dependence on $z$. In a similar vein Huang and Xu \cite{HuangXu17} studied sup norm bounds of Eisenstein series with level and obtained bounds uniform in both the spectral parameter and the level.

\subsection{Statement of results}Our first result is the following translation between uniform sup norm bounds of the Eisenstein series $E(z,s)$ and hybrid subconvexity bounds for $L_K(s, \chi)$. Let 
\begin{align}\label{F}\mathcal{F}:=\{ z\in \H\mid -1/2\leq \Re z\leq 0, |z|\geq 1\text{ \rm or }  0< \Re z< 1/2, |z|>1\},\end{align} 
 denote the standard fundamental domain for $\Gamma_0(1)$.
 \begin{thm} \label{translationthm} 
Assume the following uniform bound uniformly for all $z=x+iy\in \mathcal{F}$;
\begin{align} \label{bound1} E(z,1/2+it) \ll  y^\delta (|t|+1)^\theta, \end{align}
with $1/2\leq \delta\leq 1$ and $\theta>0$. Then it follows that
\begin{align}\label{singleL} L_K(1/2+it, \chi )\ll_{\eps} |D|^{1/4+\eps}\,  (|t|+1)^{\theta+\eps},  \end{align}
for any $\eps>0$ and $\chi\in \widehat{\Cl (K)}$, a (wide) class group character of a quadratic extension $K/\Q$ (real or imaginary) of discriminant $D$. \\
Furthermore it also follows from (\ref{bound1}) that 
\begin{align}\label{squareL}\sum_{\chi \in \widehat{\Cl (K)}} |L_K(1/2+it, \chi )|^2\ll_{\eps} |D|^{\delta+\eps}\,  (|t|+1)^{2\theta+\eps}, \end{align}
 for any $\eps>0$.
\end{thm}
The second part of this paper is concerned with proving a result of the type (\ref{bound1}). As we will see in Section \ref{extyoung} below, the results of Young \cite{Yo18} imply the following.

\begin{thm}[M. Young]\label{young}
For $z\in \mathcal{F}$, the standard fundamental domain (\ref{F}) for $\Gamma_0(1)$, we have
\begin{equation}  \label{youngbound} E(z, 1/2+it) \ll_\eps y^{1/2} (|t|+1)^{3/8+\eps},
\end{equation}
for any $\eps>0$.
\end{thm}
\begin{remark}
Huang and Xu \cite[Theorem 1.1]{HuangXu17} obtained the slightly stronger bound $E(z,s)\ll_\eps y^{1/2}+|t|^{3/8+\eps}$.
\end{remark}

It turns out however to be a much more delicate task to upgrade Blomer's Weyl type estimate to a uniform one, which is the main technical contribution of this paper. Our result is the following.
\begin{thm} \label{mainthm}
For $z\in \mathcal{F}$, the standard fundamental domain (\ref{F}) for $\Gamma_0(1)$, we have
\begin{equation}  \label{mainbound} E(z, 1/2+it) \ll_\eps y^{1/2} (|t|+1)^{1/3+\eps},
\end{equation}
for any $\eps>0$.
\end{thm}
Combining this bound with Theorem \ref{translationthm}, we arrive at the following.
\begin{cor} Let $K/\Q$ be a quadratic extension (real or imaginary) of discriminant $D$ and $\chi$ a (wide) class group character of $K$. Then 
\begin{equation}\label{L-bound} L_K(1/2+it, \chi )\ll_{\eps} |D|^{1/4+\eps}\,  (|t|+1)^{1/3+\eps}, \end{equation}
and
\begin{align}\label{sqsum2}\sum_{\chi \in \widehat{\Cl (K)}} |L_K(1/2+it, \chi )|^2\ll_{\eps} |D|^{1/2+\eps}\,  (|t|+1)^{2/3+\eps},  \end{align}
for any $\eps>0$.
\end{cor}
\begin{remark} Observe that for imaginary quadratic fields, (\ref{sqsum2}) corresponds to Lindelöf on average in the $D$-aspect, since $h(K)\gg |D|^{1/2-\eps}$. On the other hand if $K/\Q$ is a real quadratic fields with class number 1, (\ref{sqsum2}) just recovers (\ref{L-bound}).
\end{remark}
\begin{remark}
As mentioned above it has been conjectured \cite[Conjecture 3.10]{Sa95} that the following should hold for all $\eps >0$; 
 \begin{align} \label{conjecture4} \sup_{z\in C} |E(z,1/2+it)| \ll_{\eps,C} (|t|+1)^\eps, \end{align}
 where $C\subset \H$ is a compact set. This implies the Lindelöf hypothesis in the $t$-aspect for the class group $L$-function. In the last section we will speculate what the uniform analogue of (\ref{conjecture4}) should be.  \end{remark}
\subsubsection{Hybrid subconvexity bounds for class group $L$-functions}The first to obtain subconvexity for class group $L$-functions seems to be S\"{o}hne \cite{Soehne97} in the $t$-aspect and Duke, Friedlaner and Iwaniec \cite{DuFrIw02} in the $D$-aspect (which was then improved numerically by Blomer, Harcos and Michel \cite{BlHaMi07}). The first to achieve subconvexity in both aspects simultaneously (with an unspecified exponent) was Michel and Venkatesh \cite{MichelVenk10} as a consequence of their solution of the subconvexity problem for $\GL_2$ automorphic $L$-functions (for general number fields). The results of Michel and Venkatesh were then later made explicit by Wu \cite{HuAndersen18}. More precisely \cite[Corollary 1.4]{HuAndersen18} states that if $\pi$ is an automorphic representation of $\GL_2(\A_\Q)$ with (unitary) central character $\omega$, then we have
\begin{align}\label{Huesub}L(\pi,1/2)\ll \fatC(\pi)^{1/4}\left( \frac{\fatC(\pi)}{\fatC(\omega)}\right)^{-\frac{1-2\theta}{40}}\fatC(\omega)^{-1/160},\end{align}
where ${\bf C}(\pi),{\bf C}(\omega)$ denote the analytic conductors of respectively $\pi,\omega$ and $\theta$ is any approximation towards the Ramanujan--Petersson conjecture. Let us briefly explain how to extract a subconvexity bound for class group $L$-functions from (\ref{Huesub}). \\
Let $\chi$ be a (wide) class group character of the quadratic extension $K/\Q$ of conductor $D$, $\theta_\chi\in \mathcal{M}_1(\Gamma_0(|D|), \chi_D)$ the theta series associated to $\chi$ (see \cite[Section 14.3]{IwKo}) and $\pi_\chi$ the corresponding automorphic representation of $\GL_2(\A_\Q)$. The analytic conductor of the automorphic representation $\pi_\chi \otimes |\cdot|_{\A_\Q}^{it}$ is given by $D(|t|+1)^2$ and the same is true for the analytic conductor of its central character. By plugging this into (\ref{Huesub}) above, we thus get
$$L(\pi_\chi \otimes |\cdot|_{\A_\Q}^{it},1/2)=L_K(1/2+it, \chi )\ll \left(|D|^{1/4}\,  (|t|+1)^{1/2}\right)^{1-1/40}, $$
which is the state of the art for hybrid subconvexity. We observe that the bound (\ref{L-bound}) improves on this in certain regimes of $t$ and $D$. Combining the result of Wu with ours, we arrive at the following improvement. 
\begin{cor}\label{hybridsub}
Let $K/\Q$ be a quadratic extension of discriminant $D$ and $\chi$ a (wide) class group character of $K$. Then we have
\begin{equation}\label{Lbound} L_K(1/2+it, \chi )\ll_\eps \begin{cases}|D|^{1/4+\eps}\,  (|t|+1)^{1/3+\eps},& \text{for } t> |D|^{3/74}\\
\left( |D|^{1/4}\,  (|t|+1)^{1/2}\right)^{1-1/40},& \text{for } t\leq |D|^{3/74} \end{cases}, \end{equation}
for any $\eps>0$.
\end{cor}
\begin{remark}
The state of the art hybrid subconvexity bound for $\GL_1$ automorphic $L$-functions \cite[Corollary 1.2]{Wu19} is very similar to the above; the best hybrid subconvexity bound is obtained by combining the results of Wu \cite{Wu19} and those of S\"{o}hne \cite{Soehne97}. Notice that the bounds obtained in these two papers depend on the number field and are thus not relevant in our hybrid setting.     
\end{remark}
\begin{remark} In the special case where $\chi$ is a genus character, we have the following factorization in terms of quadratic Dirichlet $L$-functions;
$$ L_K(s, \chi)= L(s,  \left(\tfrac{d_1}{\cdot}\right))L(s, \left(\tfrac{d_2}{\cdot}\right)),$$
where $\chi$ corresponds to the factorization $d_1 d_2=D$. In this case it follows from \cite[(1.8)]{Young17} that we have the following improvement on the above;
$$L_K(1/2+it, \chi)\ll_\eps |D|^{1/6+\eps}(|t|+1)^{1/3+\eps}.$$
\end{remark}

\section*{Acknowledgement}
I would like to express my gratitude to my advisor Morten Risager for suggesting this problem to me and for pointing me to \cite{Bl18} and \cite{HuAndersen18}, and to the referee for many useful comments, which enabled me to prove a stronger result.

\section{From sup norm bounds to subconvexity} 

In this section we will prove Theorem \ref{translationthm}. First of all we will introduce some background on quadratic fields and the formulas due to Hecke mentioned above. 

\subsection{Quadratic fields} We will now recall a few standard facts about quadratic fields and refer to \cite[Chapter 22]{IwKo}, \cite[Section 1]{Sarnak85} and \cite[Section 2]{DuImTo} for more background.\\
Let $K/\Q$ be a quadratic extension of number fields, then we can write $K=\Q[\sqrt{D}]$ where $D$ is the discriminant of $K$. We denote by $\Cl (K)$ the class group of $K$ consisting of classes of fractional ideals modulo principal ideals. According to Gauss each fractional ideal class $\mathfrak{a}$ corresponds to an equivalence class of integral binary quadratic forms of discriminant $D$ modulo integral linear transformations. When $D<0$ we can to each $\mathfrak{a}\in\Cl (K)$ associate a Heegner point on the modular curve given by;
$$ z_\mathfrak{a}:=\frac{-b+i\sqrt{|D|}}{2a}\in X_0(1),$$
where $Q=aX^2+bXY+cY^2$ is any representative of $\mathfrak{a}$. We denote by $h(K)$ the size of the class group and we have the following (ineffective) bound due to Siegel; 
\begin{align}\label{Siegel} |D|^{1/2-\eps}\ll_\eps h(K)\ll_\eps |D|^{1/2+\eps}.\end{align}

When $D> 0$, we can analogously to any ideal class $\mathfrak{a}$ in the (wide) class group of $K$ associate a certain primitive, closed geodesic $C_\mathfrak{a}$ on $X_0(1)$. If $\mathfrak{a}$ corresponds to some integral binary quadratic form $Q=aX^2+bXY+cY^2$, then $C_\mathfrak{a}$ is defined as the projection onto $X_0(1)$ of a certain arc on the semi-circle $S_Q\subset \H$ defined by the end-points $\frac{-b\pm \sqrt{D}}{2a}$ (see the references above for the precise definition). The hyperbolic line element on $X_0(1)$ is given by $|ds|= |dz|/y$ and $C_\mathfrak{a}$ has hyperbolic length $2\log \epsilon_K$, where $\epsilon_K$ is the fundamental unit of $K$. Similar to the imaginary quadratic case we have the (ineffective) bound;
\begin{align}\label{Siegel2} |D|^{1/2-\eps}\ll_\eps h(K)\log \epsilon_K \ll_\eps |D|^{1/2+\eps}, \end{align}
also due to Siegel.

\subsection{Hecke's formula for class group $L$-functions} 
For a real or imaginary quadratic extension $K/\Q$ and a character $\chi$ of $\Cl (K)$, we associate the class group $L$-function absolutely convergent for $\Re s>1$;
\begin{align}
\label{classgroupL} L_K(s,\chi):= \sum_{\mathfrak{a}} \chi(\mathfrak{a}) N_K(\mathfrak{a})^{-s}=\prod_{\mathfrak{p}}(1-\chi(\mathfrak{p})N_K(\mathfrak{p})^{-s}),
\end{align}    
where $N_K$ is the norm and the sum runs over all integral ideals of $K$ and the product is taken over integral prime ideals of $K$. The class group $L$-functions admit analytic continuation and functional equations, which we will see shortly follows from the same properties for the non-holomorphic Eisenstein series.\\

The connection between class group $L$-functions and Eisenstein series is given by a beautiful formula due to Hecke. In the introduction we already mentioned that for imaginary quadratic extensions $K/\Q$, the formula reads \cite[(22.58)]{IwKo};   
\begin{align*}
 L_K(s, \chi)= \frac{2^{s+1} \zeta(2s) |D|^{-s/2}}{\omega_K } \sum_{\mathfrak{a}}\chi(\mathfrak{a}) E(z_{\mathfrak{a}},s),
 \end{align*}
where the sum runs over a complete set of representatives for the class group of $K$, $z_{\mathfrak{a}}$ is the associated Heegner point and $\omega_K\in\{2,4,6\}$ denotes the number of roots of unity in $K$.\\ 
For real quadratic fields, we have similarly the following formula \cite[(7.7)]{DuImTo};
\begin{align}
 \label{heegnerintegral} L_K(s, \chi)= \frac{\zeta(2s) D^{-s/2} \Gamma(s)}{\Gamma(s/2)^2} \sum_{\mathfrak{a}} \chi(\mathfrak{a})  \int_{C_\mathfrak{a}} E(z,s)y^{-1}|dz|.
 \end{align}
We observe that analytic continuation and functional equation for $L_K(s, \chi)$ now follows from the corresponding properties of the Eisenstein series \cite[Theorem 6.5]{Iw}. 

\subsection{Proof of Theorem \ref{translationthm}} 
In this section we will prove Theorem \ref{translationthm}. To do this we will need a lemma that bounds averages over Heegner points (resp. cycles) of the function $y:X_0(1)\rightarrow \R_+$ defined by $y(z):=\Im (z_\mathcal{F})$, where $z_\mathcal{F}\in \H$ is the representative of $z\in X_0(1)$ which lies in $\mathcal{F}$, the standard fundamental domain (\ref{F}) for $\Gamma_0(1)$. Observe that this function is continuous. 
\begin{lemma}\label{boundcongruence}
Let $K/ \Q$ be a quadratic field of discriminant $D$. Then we have for any $\delta>0$ and $\eps>0$;
$$  \sum_{\mathfrak{a}\in \Cl(K)} \begin{cases}   y(z_\mathfrak{a})^\delta & \text{\rm if } D<0,\\ \int_{C_\mathfrak{a}} y(z)^\delta \, |ds| &\text{\rm if } D>0,\end{cases}  \ll_\eps |D|^{\max(\delta, 1)/2+\eps}. $$
\end{lemma}
\begin{proof}
Assume $D<0$. The representative of $z_\mathfrak{a}\in X_0(1)$ which lies in $\mathcal{F}$, is exactly given by
$$ (z_\mathfrak{a})_\mathcal{F}=\frac{-b+i\sqrt{|D|}}{2a},  $$
where the integral binary quadratic form $aX^2+bXY+cY^2$ of discriminant $D$ corresponds to $\mathfrak{a}$ and $(a,b,c)$ is reduced \cite[(22.12)]{IwKo}, meaning that;
$$  -a<b\leq a\leq c\quad \text{or} \quad -a\leq b\leq a=c.  $$
Since $\mathcal{F}\subset \{z\in \H\mid \Im z\geq \sqrt{3}/2\}$, we conclude that $a\ll \sqrt{|D|}$ and thus we get;
\begin{align*}  
\sum_{\mathfrak{a}\in \Cl(K)} y(z_\mathfrak{a})^\delta =& |D|^{\delta/2}\sum_{a>0} \frac{\# \{a,b,c\mid b^2-4ac=D, (a,b,c)\text{ reduced} \}}{(2a)^\delta}\\
\ll& |D|^{\delta/2}\sum_{0<a\ll |D|^{1/2}}   \frac{\rho_D(a)}{a^\delta},
\end{align*}
where $\rho_D(a)=\# \{ 0<b\leq 2a \mid b^2\equiv D\modulo 4a\}$. It is well-known \cite[p. 521]{IwKo} that $\rho_D$ is multiplicative with $\rho_D(p^\alpha)=1+\chi_D(p)$ if $p\not| D$, $\rho_D(p)=1$ if $p|D$ and $\rho_D(p^\alpha)=0$ if $p|D$, $\alpha>1$, which implies the bound $\rho_D(a)\ll \sum_{d|a}1 \ll_\eps a^{\eps} $. Thus we conclude that
$$ \sum_{\mathfrak{a}\in \Cl(K)} y(z_\mathfrak{a})^\delta \ll_\eps |D|^{\delta/2} \sum_{0<a\ll \sqrt{|D|}} \frac{a^\eps}{a^\delta}\ll |D|^{1/2\max(\delta,1)+\eps}, $$
as wanted.\\

Now we turn to the case $D>0$. We denote by $\Omega_D$ all integral binary quadratic forms of discriminant $D$ and for $Q=aX^2+bXY+cY^2\in \Omega_D$, we denote by $S_Q$ the semi-circle in $\H$ with end-points $\frac{-b\pm\sqrt{D}}{2a}$. Then it follows from an easy lemma \cite[Lemma 6]{DuImTo11} (observe that they use a different looking but equivalent measure) that;
\begin{align} 
\label{furtherequal}\sum_{\mathfrak{a}\in \Cl(K)} \int_{C_\mathfrak{a}} y(z)^\delta \, |ds|=\sum_{Q\in \Omega_D}  \int_{S_Q \cap \mathcal{F}} y(z)^\delta \, |ds|,
\end{align}
where $\mathcal{F}$ is the standard fundamental domain (\ref{F}) for $\Gamma_0(1)$.\\ 
Now we take the quotient from the left by $\Gamma_\infty=\langle T \rangle$ where $T=\begin{psmallmatrix} 1&1\\0&1\end{psmallmatrix}$, which rewrites (\ref{furtherequal}) as the following;
\begin{align}\label{finfty}\sum_{[Q]\in \Gamma_\infty\backslash\Omega_D}  \int_{S_Q \cap \mathcal{F}^{(\infty)}} y(z)^\delta\,  |ds|,\end{align}
where $\mathcal{F}^{(\infty)}:=\cup_{n\in \Z} T^{(n)}\mathcal{F}$ is the union of all horizontal translates of $\mathcal{F}$ (notice that the integral above does not depend on the choice of $Q$). Since $\mathcal{F}^{(\infty)}\subset \{z\in \H\mid \Im z\geq \sqrt{3}/2\}$, we only get contributions in (\ref{finfty}) from quadratic forms $Q=aX^2+bXY+cY^2$ with $a\ll \sqrt{D}$ and furthermore we can pick representatives of $\Gamma_\infty\backslash\Omega_D$ satisfying $|b|\leq 2a$. Now we recall that $|ds|=y^{-1}|dz|$ and use the trivial fact that the Euclidean circumference of $S_Q$ is $\ll \frac{D^{1/2}}{a}$, which implies;
\begin{align*}
\sum_{[Q]\in \Gamma_\infty\backslash\Omega_D}  \int_{S_Q \cap \mathcal{F}^{(\infty)}} y(z)^\delta\, |ds|&= \sum_{0<a\ll D^{1/2}}  \sum_{\substack{[Q]\in \Gamma_\infty\backslash\Omega_D,\\ Q(1,0)=a}}\int_{S_{Q}\cap \mathcal{F}^{(\infty)}} y(z)^{\delta-1} |dz|\\
&\ll \sum_{0<a\ll D^{1/2}}  \sum_{\substack{[Q]\in \Gamma_\infty\backslash\Omega_D,\\ Q(1,0)=a}} \frac{D^{1/2}}{a}\left( \max_{z\in S_{Q}\cap \mathcal{F}^{(\infty)}} y(z)^{\delta-1}\right)\\
 &\ll D^{1/2+\max(\delta-1,0)/2}\sum_{0<a\ll D^{1/2}}  \frac{\rho_D(a)}{a}.
\end{align*}
Now the conclusion follows exactly as in the case of negative $D$ using the bound $\rho_D(a)\ll_\eps a^\eps$ (which also holds for $D>0$ by the above).
\end{proof}

Now we are ready to prove Theorem \ref{translationthm}.
\begin{proof}[Proof of Theorem \ref{translationthm}]
Consider the case $D<0$. By feeding (\ref{bound1}) into (\ref{heegnersum}), we see that
\begin{align}
L_K(1/2+it,\chi ) \ll_\eps  \frac{(|t|+1)^\eps}{|D|^{1/4}} \sum_{\mathfrak{a}}  y(z_\mathfrak{a})^{\delta} (|t|+1)^{\theta},
\end{align}
where we used some standard estimates for $\zeta$ on $\Re s=1$. \\
Now since we assumed $\delta\leq 1$, it follows from Lemma \ref{boundcongruence} that
$$L_K(1/2+it,\chi) \ll_{\eps} |D|^{1/4+\eps}  (|t|+1)^{\theta+\eps}.$$
as wanted. \\

To prove (\ref{squareL}), we observe that by orthogonality, the formula (\ref{heegnersum}) implies that 
$$ \sum_\chi |L_K(1/2+it, \chi)|^2   = \frac{8 h(K) |\zeta(1+2it)|^2}{\omega_K^2 |D|^{1/2}}\sum_\mathfrak{a} |E(z_\mathfrak{a},1/2+it)|^2.  $$
Thus by the assumption (\ref{bound1}), Siegel's bound (\ref{Siegel}) and standard estimates for the zeta function, we get
$$  \sum_\chi |L_K(1/2+it, \chi)|^2 \ll_\eps (|t|+1)^{2\theta+\eps} |D|^\eps \sum_\mathfrak{a}y(z_\mathfrak{a})^{2\delta},$$
and the result follows directly from Lemma \ref{boundcongruence}.\\

The proof of (\ref{singleL}) for $D$ positive is exactly the same using Lemma \ref{boundcongruence} and Hecke's formula (\ref{heegnerintegral}) in the case $D>0$. \\
In order to prove (\ref{squareL}), we use orthogonality as above to get
\begin{align*}
\sum_\chi |L_K(1/2+it, \chi)|^2 \ll_\eps  (|t|+1)^{2\theta+\eps} \frac{h(K)}{D^{1/2}} \sum_\mathfrak{a}\left |\int_{C_\mathfrak{a}} y(z)^\delta |ds|\right|^2 .
\end{align*}
Now we apply Cauchy-Schwarz to bound the above by 
$$ (|t|+1)^{2\theta+\eps} \frac{h(K)\log \epsilon_K}{D^{1/2}} \sum_\mathfrak{a} \int_{C_\mathfrak{a}} y(z)^{2\delta}|ds|,  $$
and the results follows from Lemma \ref{boundcongruence} and Siegel's bound (\ref{Siegel2}).
 \end{proof}

\begin{remark} If one believes the sup norm conjecture (\ref{conjecture4}), Theorem \ref{translationthm} tells you in particular that the cancellations in individual Eisenstein series are strong enough to give the Lindelöf hypothesis for class group $L$-functions in the $t$-aspect. It is however conjectured that (\ref{eissup}) holds for eigenfunctions on any hyperbolic surface \cite[Conjecture 3.10]{Sa95}. So in some sense the $t$-aspect is not essentially arithmetic. This method is however not able to give subconvexity estimates in the $D$-aspect for individual $L$-functions. This is due to the fact that the sup norm bounds do not \lq \lq see{\rq \rq} the arithmetics of the Heegner points (it is uniform for $z$ in a fixed compact set) and the cancellation between Eisenstein series evaluated at the different Heegner points is exactly what gives rise to subconvexity behavior in the $D$-aspect. In the last section (see (\ref{conjecture})), we will state a uniform analogue of the conjecture (\ref{eissup}), which using (\ref{squareL}) does give Lindelöf on average in the $D$-aspect for imaginary quadratic fields.\end{remark}

\section{Uniform sup norm bounds of Eisenstein series} 

In this section we will prove the hybrid bound (\ref{youngbound}) and (\ref{mainbound}) for the classical Eisenstein series. The proof of (\ref{youngbound}) follows directly from \cite{Yo18}. The proof of (\ref{mainbound}) requires much more work and is an adaptation (and elaboration) of the argument in \cite{Bl18} building on \cite{Ti34}, which in turn is an extension of the van der Corput method \cite[Section 8.3]{IwKo}.
 
 \subsection{Uniform bounds for Eisenstein series following Young}\label{extyoung} 
In \cite{Yo18} Young extends the method used by Iwaniec and Sarnak in \cite{IwSa} to give the first non-trivial result towards the sup norm conjecture for the modular curve. The main insight of Young was that one can choose a more efficient mollifier, which improves the bound for the continuous spectrum. The method of Iwaniec and Sarnak embeds respectively the cusp form and Eisenstein series into the entire spectrum of the modular curve. Then an application of the Selberg trace formula (with a carefully chosen test function) reduces the sup norm bound to a bound of the geometric side, which can be done with elementary means. The action of the Hecke operators plays a crucial role in the argument.\\
In \cite{Yo18} the sup norm bound is stated as a bound in the $t$-aspect with $z$ in a fixed compact set, but as Young also mentions the method yields something slightly stronger (this was also observed by Huang and Xu \cite[p. 2]{HuangXu17}).\\ 
The main inequality in Young's paper is \cite[(6.3)]{Yo18}, which gives
\begin{align}
|E(z, 1/2+it)|^2 \ll_\eps (N|t|)^\eps \left( \frac{|t|}{N}+|t|^{1/2}(N+N^{1/2}y) \right),
\end{align}
where $N$ is some parameter to be chosen appropriately. By inspecting \cite[Lemma 4.1, Lemma 5.1]{Yo18} one sees that the restrictions on the variables are $\log N\gg (\log t)^{2/3+\delta}$ for some fixed $\delta>0$ and $y\ll |t|^{100}$. In particular in the range $y\ll |t|^{1/4}$, we can put $N=|t|^{1/4}$ and get
$$  |E(z, 1/2+it)|^2 \ll_\eps |t|^{3/4+\eps}+ |t|^{3/4+\eps}+ |t|^{5/8+\eps}y.   $$
From this we conclude
$$ |E(z, 1/2+it)| \ll_\eps  y^{1/2} |t|^{3/8+\eps} , \quad 1\ll y\ll |t|^{1/4}.$$
In the range $y\gg |t|^{1/4}$, we have the trivial bound \cite[(3.2)]{Yo18}, which yields
$$    |E(z, 1/2+it)|\ll_\eps y^{1/2}+ |t|^{3/8+\eps}. $$
Combining the two, concludes the proof of Theorem \ref{young}.

\subsection{Titchmarsh's method for bounding Epstein zeta functions}  
 Now we turn to the proof of Theorem \ref{mainthm}. The following serves first of all as an extension of Blomer and Titchmarsh's work but secondly as an elaboration of some of the details, which are left out in \cite{Bl18}. The approach expresses the non-holomorphic Eisenstein series in terms of an {\it Epstein zeta function}, which is then  bounded using the van der Corput method from the theory of exponential sums. Originally Titchmarsh considered only Epstein zeta functions associated to diagonal matrices and there are some technical difficulties to deal with general Epstein zeta functions. Furthermore in order to get a bound uniform in the entries of the matrix defining the Epstein zeta function, one has to modify parts of the argument.\\
 
Given any positive definite matrix $Z\in \GL_2(\R)$, we can consider the quadratic form $Q(\x)=\x \, Z\, \x ^T$, $\x=(x_1,x_2)\in \R^2$ and the associated Epstein zeta function
$$ E_{\text{\rm Epstein}}(Z, s):= \sum_{ \x \in \Z^2 \backslash (0,0)} Q(\x)^{-s}, $$
which satisfies the functional equation
$$ \Gamma_\R(2s)E_{\text{\rm Epstein}}(Z, s)= (\det Z)^{-1/2} \Gamma_\R(2(1-s))E_{\text{\rm Epstein}}(Z^{-1},1-s),$$
where $\Gamma_\R (s):=\pi^{-s/2} \Gamma(s/2)$.\\
Recall that this is related to the non-holomorphic Eisenstein series as follows 
\begin{equation} \zeta(2s)E(z,s)= y^s E_{\text{\rm Epstein}}(Z,s),\qquad Z=\begin{pmatrix}x^2 +y^2 & x \\ x & 1 \end{pmatrix}, \end{equation}
which reduces the sup norm problem for Eisenstein series to bounding the Epstein zeta function. We may restrict to the case where $z\in \mathcal{F}$, the standard fundamental domain (\ref{F}) for $X_0(1)$, which corresponds to considering only matrices of the form
$$ Z=\begin{pmatrix} a & b \\ b & 1 \end{pmatrix},$$ 
where $a\geq 1$ and $|b|\leq 1/2$.\\
The trivial estimate \cite[(3.2)]{Yo18};
$$E(z,1/2+it)\ll y^{1/2}+(t/y)^{1/2}$$
yields (\ref{mainbound}) in the range $|t|^{1/6}\ll y$ and thus in the sequel we may assume $a\ll |t|^{1/3}$ and thus also $|t|\gg 1$.

\subsection{Reduction to an exponential sum} 
As in \cite{Bl18} we start by applying an approximate functional equation \cite[Theorem 5.3]{IwKo} with $G(u)=e^{u^2}$, but deviate slightly by using a balanced version (corresponding to putting $X=a^{1/2}$ in \cite[Theorem 5.3]{IwKo}). By estimating the contribution coming from the pole of $E_{\text{\rm Epstein}}(Z,s)$ at $s=1$ trivially, the approximate functional equation yields
\begin{align}\nonumber E_{\text{\rm Epstein}}(Z, 1/2+it)&= \sum_{{\bf x}\neq 0} \frac{W^+_t(Q_+({\bf x})a^{-1/2})}{Q_+({\bf x})^{1/2+it}}\\
\label{approx}& + \frac{\Gamma_\R (1-2it)}{\Gamma_\R(1+2it) (\det Z)^{1/2}} \sum_{{\bf x}\neq 0} \frac{W^{-}_t(Q_-({\bf x})a^{1/2})}{Q_-({\bf x})^{1/2-it}} +O(1)\end{align}
where $Q_\pm({\bf x})={\bf x}\, Z^{\pm 1} \, {\bf x}^T$ and 
$$ W^{\pm}_t(y)=\frac{1}{2\pi i}\int_{(1)} e^{u^2} \frac{\Gamma_\R(2(u+1/2\pm it))}{\Gamma_\R(2(1/2\pm it))} y^{-u}\frac{du}{u}.$$ 
The weight $W^\pm_t$ can be nicely bounded as follows; we move the contour to the line $(A)$ with $A>0$ and bound the integrand using Stirling's approximation as follows;
\begin{align*}  e^{u^2/2}\frac{\Gamma_\R(2(u+1/2\pm it))}{\Gamma_\R(2(1/2\pm it))} u^{-1}\ll \frac{e^{A^2/2}e^{-b^2/2} \pi^{-A/2}e^{-A} (|t|^A + (b+A)^A)}{A+|b|} \ll_A |t|^A ,
\end{align*}
with $u=A+ib$ using that $e^{-b^2/2}  (b+A)^A \rightarrow 0$ as $b\rightarrow \infty$. Thus we get the bound
$$ W^\pm_t(y)\ll_A |t|^A/y^A \int_{-\infty}^\infty e^{-x^2/2} dx \ll |t|^A/y^A, $$
and more generally one deduces $\frac{\partial^n }{\partial y^n}W^\pm_t(y)\ll_A |t|^A/y^{A+n}$ as in \cite[Proposition 5.4]{IwKo}. From this we see that the contributions in (\ref{approx}) from $\x$ such that $Q_\pm(\x)\gg a^{\pm 1/2} |t|^{1+\eps}$ are negligible. \\

To deal with the remaining sums in (\ref{approx}), we divide the range of summation into dyadic rectangles of the form $(X_1,2X_1)\times (X_2,2X_2)$. Observe that we get  $O(\log^2 t)$ such rectangles, which implies that it suffices to bound each of these dyadic sums individually. \\  
For each such rectangle we get by two-dimensional partial summation;
\begin{align}  \label{dyadic}\sum_{\substack{X_1\leq x_1\leq 2X_1\\ X_2\leq x_2\leq 2X_2}} \frac{W^+_t(Q_+({\bf x})a^{-1/2})}{Q_+({\bf x})^{1/2+it}} =  F_+(2\X) \sum_{\substack{X_1\leq x_1\leq 2X_1\\ X_2\leq x_2\leq 2X_2}} e^{i t \log Q_+(\x)} \\
\nonumber -\int_{X_1}^{2X_1} \left( \sum_{\substack{X_1\leq x_1\leq x \\ X_2\leq x_2\leq 2X_2}}e^{it\log Q_+(\x)}\right) F_+^{(1,0)} (x, 2X_2) dx\\
\nonumber -\int_{X_2}^{2X_2} \left( \sum_{\substack{X_1\leq x_1\leq 2X_1 \\ X_2\leq x_2\leq y}}e^{it\log Q_+(\x)}\right) F_+^{(0,1)} (2X_1, y) dy\\
\nonumber + \int_{X_1}^{2X_1}\int_{X_2}^{2X_2}   \left( \sum_{\substack{X_1\leq x_1\leq x \\ X_2\leq x_2\leq y}}e^{it\log Q_+(\x)}\right)F_+^{(1,1)}(x_1,x_2)\, dxdy, 
\end{align}
where $\X=(X_1,X_2)$, $F_+({\bf x})=W^+_t(Q_+({\bf x})a^{-1/2})/Q_+( {\bf x} )^{1/2}$ and $F_+^{(i,j)}:=\frac{\partial^{i+j}F_+}{\partial x_1^i \partial x_2^j }$.\\ 
Similarly we get 
\begin{align}\label{dyadic-}\sum_{\substack{X_1\leq x_1\leq 2X_1\\ X_2\leq x_2\leq 2X_2}} \frac{W^-_t(Q_-({\bf x})a^{1/2})}{(\det Z)^{1/2}Q_-({\bf x})^{1/2-it}}=F_-(2\X)\sum_{\substack{X_1\leq x_1 \leq 2X_1 \\ X_2\leq x_2\leq 2X_2}}e^{it\log Q_-(\x)} +\ldots,  \end{align}
where $F_-({\bf x})=W^-_t(Q_-({\bf x} )a^{1/2})/\left((\det Z)Q_-( {\bf x} )\right)^{1/2}$.\\ 
Now we have reduced the desired bound on the Epstein zeta function to proving a certain estimate on exponential sums. The result we need is the following.

\begin{prop} \label{mainprop}For $\X=(X_1,X_2)$ satisfying $Q_+(\X)\ll a^{1/2}|t|^{1+\eps}$, we have the following bound;
\begin{equation}\label{mainpropeq} \frac{1}{Q_+(\X)^{1/2}}\sum_{\substack{X_1 \leq x_1 \leq X_1' \\ X_2 \leq x_2 \leq X_2'}} e^{it\log Q_+(\x)}\ll_\eps |t|^{1/3+\eps}, \end{equation}
uniformly in $a\geq 1$, where $X_i\leq X_i'\leq 2X_i$. Similarly for $\X=(X_1,X_2)$ satisfying $Q_-(\X)\ll a^{-1/2} |t|^{1+\eps}$, we have 
\begin{equation}\label{mainpropeq2}\frac{1}{((\det Z)Q_-(\X))^{1/2}}\sum_{\substack{X_1 \leq x_1 \leq X_1' \\ X_2 \leq x_2 \leq X_2'}} e^{it\log Q_-(\x)}\ll_\eps |t|^{1/3+\eps}, \end{equation}
where $X_i\leq X_i'\leq 2X_i$.
\end{prop}
\begin{remark}Observe that when proving (\ref{mainpropeq}), we may assume 
\begin{equation}\label{lowerbound} X_1\gg |t|^{1/3}\quad \text{ and }\quad X_2\gg |t|^{1/3} a^{1/2}, \end{equation}
and similar when proving (\ref{mainpropeq2}), we may assume
\begin{equation}\label{lowerbound2} X_1\gg |t|^{1/3} a^{1/2}\quad \text{ and }\quad X_2\gg |t|^{1/3}, \end{equation}
since otherwise the bounds follows from the trivial estimate on the exponentials. \end{remark}

Now let us see how Theorem \ref{mainthm} follows from the above proposition.
\begin{proof}[Proof of Theorem \ref{mainthm} assuming Proposition \ref{mainprop}] We will begin by deducing from Proposition \ref{mainprop} that $E_{\text{\rm Epstein}}(Z,s)\ll_\eps (|t|+1)^{1/3+\eps}$ for all $Z$ as above; by the above reductions, it suffices to prove the same bound for each of the dyadic sums (\ref{dyadic}) and (\ref{dyadic-}) with $X_1,X_2$ satisfying respectively $Q_\pm(\X)\ll a^{\pm 1/2} |t|^{1+\eps}$. We do this by bounding each of the four terms, we get after applying partial summation separately (observe that we may assume $|t|\gg 1$).\\ 
The above estimates for $W^+_t$ imply $ W^+ _t(Q_+(\x)a^{- 1/2})\ll |t|^\eps $, which together with (\ref{mainpropeq}) implies that we can bound the first sum on the right-hand side of (\ref{dyadic}) by the following;
$$F_+(2\X) \sum_{\substack{X_1\leq x_1\leq 2X_1\\ X_2\leq x_2\leq 2X_2}} e^{i t\log Q_+(\x)}\ll |t|^{1/3+\eps}.$$
Similarly using $\frac{\partial^n }{\partial y^n}W^+_t(y)\ll |t|^{A}/y^{A+n}$ and the chain rule, we get
$$ F_+^{(1,0)}(\x) \ll \frac{|t|^{\eps}a^{1/2}}{Q_+(\X)}, \quad F_+^{(0,1)}(\x) \ll \frac{|t|^{\eps} }{Q_+(\X)}, \quad  F_+^{(1,1)}(\x) \ll \frac{|t|^{\eps}a^{1/2}}{Q_+(\X)^{3/2}},$$
which together with (\ref{mainpropeq}) implies
\begin{align*} \int_{X_1}^{2X_1} \left( \sum_{\substack{X_1\leq x_1\leq x \\ X_2\leq x_2\leq 2X_2}}e^{it\log Q_+(\x)}\right) F_+^{(1,0)} (x, 2X_2) dx\\ 
\ll  X_1 |t|^{1/3+\eps}Q_+(\X)^{1/2} \frac{a^{1/2}}{Q_+(\X)} \ll |t|^{1/3+\eps},\end{align*}
using $X_1a^{1/2}\ll Q_+(\X)^{1/2}$, and similarly for the other one-dimensional integral. Finally a similar calculation gives
\begin{align*} \int_{X_1}^{2X_1}\int_{X_2}^{2X_2}   \left( \sum_{\substack{X_1\leq x_1\leq x \\ X_2\leq x_2\leq y}}e^{it\log Q_+(\x)}\right)F^{(1,1)} (x, y)\, dxdy&\\
\ll \frac{X_1X_2 a^{1/2}Q_+(\X)^{1/2}|t|^{1/3+\eps}}{Q_+(\X)^{3/2}}&,  \end{align*}
which yields the desired bound for the $Q_+$-sum.\\
The sum involving $Q_-$ can be bounded similarly using 
\begin{align*} F_-^{(1,0)}(\x) \ll \frac{|t|^{\eps}}{(\det Z)Q_-(\X)}, \quad F_-^{(0,1)}(\x) \ll \frac{  |t|^{\eps}a^{1/2}}{(\det Z)Q_-(\X)}, \\
  F_-^{(1,1)}(\x) \ll \frac{|t|^{\eps} a^{1/2}}{\left( (\det Z)Q_-(\X)\right)^{3/2}},\end{align*}
which yields the desired bound for the Epstein zeta function.\\
Thus we conclude that 
$$ E(z,1/2+it)=\frac{y^{1/2+it}}{\zeta(1+2it)}E_{\text{\rm Epstein}}(Z,1/2+it)\ll_\eps y^{1/2}(|t|+1)^{1/3+\eps},$$
using $\zeta(1+2it)\gg_\eps (|t|+1)^{-\eps}$. This finishes the proof.
 \end{proof} 

\section{A uniform bound for an exponential sum in two variables} 
In this section we will prove Proposition \ref{mainprop} using an extension of the ideas of Titchmarsh and Blomer building on the work of van der Corput. \\ Firstly we will make a simplification; if we multiply with the phase $(\det Z)^{it}$ in (\ref{mainpropeq2}), the summands become;
$$e^{it \log (\det Z)}e^{it\log Q_-(\x)} = e^{it\log \left( (\det Z)Q_-(\x)\right)},$$ 
where $(\det Z)Q_-(\x)=x_1^2-2bx_1x_2+ax_2^2$. Since $\det Z\asymp a$, the ranges $Q_+(\X)\ll a^{1/2} |t|^{1+\eps}$ and $(\det Z)Q_-(\X)\ll (\det Z)a^{-1/2} |t|^{1+\eps}$ are the same just with $X_1$ and $X_2$ interchanged. Thus by symmetry the two bounds (\ref{mainpropeq}) and (\ref{mainpropeq2}) are equivalent, which is exactly why we used a balanced approximate functional equation in the first place.\\ 
Thus we see that it suffices to prove (\ref{mainpropeq}) under the assumption $Q_+(\x)\ll a^{1/2}|t|^{1+\eps}$. To lighten notation, we put $Q:=Q_+$. 
\subsection{Some lemmas of Titchmarsh} 
Titchmarsh \cite{Ti34} extended the van der Corput method for bounding exponential sums \cite[Section 8.3]{IwKo} to two-dimensional sums. In this section we will quote some lemmas due to Titchmarsh, which we will employ later.\\
Through-out this section we assume that 
$$f: (X_1, X_1')\times (X_2, X_2')\rightarrow \R$$ 
has algebraic partial derivatives of order one to three. We will as above use the notation $f^{(i,j)}:= \frac{\partial^{i+j}f}{\partial x_1^i\partial x_2^j}$. \\

The first lemma is a version of Weyl differencing in the two-dimensional setting. 
\begin{lemma}[Lemma $\beta$, \cite{Ti34}] \label{Wdiff}
Let $\rho \leq \min (X_1'-X_1, X_2'-X_2)$ be a positive integer. Then we have
\begin{align}\nonumber \sum_{\substack{X_1 \leq x_1\leq X_1'\\ X_2\leq x_2 \leq X_2'}} e^{if(\x)} \ll  &\frac{(X_1'-X_1)(X_2'-X_2)}{\rho}\\
\nonumber&+ \frac{(X_1'-X_1)^{1/2}(X_2'-X_2)^{1/2}}{\rho} \left(\sum_{\substack{1\leq \mu_1\leq \rho-1\\0\leq \mu_2\leq \rho-1}}  |S_1(\m)|  \right)^{1/2} \\
\label{beta}&+ \frac{(X_1'-X_1)^{1/2}(X_2'-X_2)^{1/2}}{\rho} \left(\sum_{\substack{0\leq \mu_1\leq \rho-1\\1 \leq \mu_2\leq \rho-1}}  |S_2(\m)|  \right)^{1/2},  \end{align}
where $\x=(x_1,x_2)$, $\m=(\mu_1,\mu_2)$ and
$$S_1(\m)=\sum_{\substack{X_1\leq x_1\leq X_1'-\mu_1\\X_2\leq x_2 \leq  X_2'-\mu_2}}e^{i[f(\x+\m)-f(\x)]}, \quad S_2(\m)=\sum_{\substack{X_1\leq x_1\leq X_1'-\mu_1\\X_2+\mu_2\leq x_2 \leq  X_2'}}e^{i[f(\x+(\mu_1,-\mu_2))-f(\x)]}.$$
\end{lemma}

The above lemma reduces the task to bounding the sums $S_1(\m)$ and $S_2(\m)$ with $\mu_1,\mu_2$ in the appropriate ranges. The idea of the van der Corput method is to reduce the bound of the sums $S_1(\m)$ and $S_2(\m)$ to bounding a certain integral. We will use the following extension of van der Corput's result due to Titchmarsh.

\begin{lemma}[Lemma $\gamma$, \cite{Ti34}] \label{sumint}
Let $l=\max(X_1'-X_1, X_2'-X_2)$ and assume that $f$ satisfies
$$ |f^{(1,0)}(\x)|\leq \frac{3\pi}{2},\qquad  |f^{(0,1)}(\x)|\leq \frac{3\pi}{2}. $$
Then
\begin{align}
\sum_{\substack{X_1\leq x_1\leq X_1'\\ X_2\leq x_2\leq X_2'}} e^{if(\x)}= \int_{(X_1,X_1')\times(X_2, X_2')} e^{if(\x)}d\x+ O(l).
\end{align}
\end{lemma}

Finally we gonna bound this integral by a second derivative test. 

\begin{lemma}[Lemma $\epsilon$, \cite{Ti34}] \label{vanderC}
Let $\Omega\subset \R^2$ be a rectangle and $l$ its maximal side length. If $f:\Omega\rightarrow \R$ is a function satisfying the conditions mentioned in the beginning of the section and
\begin{align} \label{normbound} r\ll |f^{(2,0)}(\x )| \ll r ,\quad r\ll |f^{(0,2)}(\x)| \ll r, \quad |f^{(1,1)}(\x)|\ll r  \\
\label{detbound} |f^{(2,0)}(\x)f^{(0,2)}(\x)-(f^{(1,1)}(\x))^2| \gg  r^2,\qquad \x \in \Omega. \end{align}
Then
$$  \int_{\Omega} e^{if(\x)}d\x \ll \frac{1+\log l+\log r}{r}, $$
where the implied constant depends only on the angle of the rectangle to the coordinate axes.  
\end{lemma} 
\begin{remark}Note that as stated, \cite[Lemma $\epsilon$]{Ti34} (or more precisely Lemma $\delta$) assumes that 
$$ |f^{(2,0)}(\x )|, |f^{(0,2)}(\x )|\geq r,\quad |f^{(2,0)}(\x)f^{(0,2)}(\x)-(f^{(1,1)}(\x))^2| \geq r^2,$$
that is; without an implicit constant in the lower bounds. By inspecting the proof of \cite[Lemma $\epsilon$]{Ti34}, one however sees that Lemma \ref{vanderC} as stated above follows with the exact same proof (this observation is also implicit in \cite{Bl18}).\end{remark}
\subsection{Applying the lemmas} 
With these results of Titchmarsh at our disposal, we are now ready to make some reductions in the direction of proving (\ref{mainpropeq}).\\ 
By applying Lemma \ref{Wdiff} with $f(\x)=t\log Q(\x)$ and $Q(\x)=ax_1^2+2bx_1x_2+x_2^2$ to the left hand side of (\ref{mainpropeq}), we reduce the task to bounding sums of the following kind; 
\begin{align}\label{S1}  S'(\m)=\sum_{\substack{X_1 \leq x_1 \leq X_1' \\ X_2 \leq x_2 \leq X_2'}} e^{ig_{\m}(\x )}, \end{align}
where 
\begin{align}\label{gmu}g_{\m}(\x):=t (\log Q(\x+\m)-\log Q(\x )),\end{align} 
$X_i'\leq 2X_i$ and $\m=(\mu_1,\mu_2)\in [0,\rho]\times [0,\rho ]$ with $\rho= o(\min(X_1,X_2))$ to be chosen appropriately later.   \\

The first step is to divide the rectangle of summation in $S'(\m)$ into rectangles $\Delta_{p,q}$ (where $p,q$ runs through an appropriate indexing set) each with side lengths $l_1\times l_2$, where
\begin{align}\label{l} l_1 \asymp \frac{Q(\X)^{3/2}} {a |t|^{1+2\eps} Q(\m)^{1/2}}, \quad l_2 \asymp \frac{Q(\X)^{3/2}}{a^{1/2} |t|^{1+2\eps} Q(\m)^{1/2}}.\end{align}  
We denote the sub-sum associated to $\Delta_{p,q}$ by $S_{p,q}(\m)$ and observe that the number of such sub-sums is bounded by; 
$$\frac{X_1X_2}{l_1 l_2}\ll  \frac{X_1X_2}{a^{-3/2}Q(\x)^3 |t|^{-2-2\eps} Q(\m)^{-1}}.$$ 
We will bound the sub-sums $S_{p,q}(\m)$ individually.
\begin{remark}
There is some balancing in choosing the values $l_1,l_2$; one the hand $l_1,l_2$ have to be small enough so that $g_{\m}$ and its derivatives are close to being constant in $\Delta_{p,q}$ (i.e. the variation is small), and on the other hand the number of rectangles $\Delta_{p,q}$ grows reciprocally with $l_1,l_2$. The reason for choosing these specific values will become clear later.
\end{remark}
\subsection{Bounds on derivatives of $g_{\m}$} 

In this subsection we will prove upper bounds on partial derivates of $g_{\m}$ and a lower bound on the determinant of the Hesse-matrix of $g_{\m}$. Titchmarsh \cite{Ti34} only considers diagonal matrices and the fact that $b\neq 0$ creates some minor technical difficulties, which were also addressed by Blomer in \cite{Bl18}. We need to be a bit more careful since we need to consider the $a$-dependence as well, so our methods of computation differ a bit from those in \cite{Bl18}; to handle the upper bounds on the derivates we apply a Taylor expansion around $\mu$ and to lower bound the Hesse determinant we use an explicit calculation. \\
First of all we will need the following lemma.
\begin{lemma}\label{boundlemma} Let $f(\x)=t\log(Q(\x))$ with $Q(x_1,x_2)=ax_1^2+2bx_1x_2+x_2^2$ where $|b|\leq 1/2$ and $a\geq 1$. Then we have
\begin{align}\label{derivativef} f^{(i,j)}(\x) \ll_{i,j} \frac{a^{i/2}|t|}{Q(\x)^{(i+j)/2}},   \end{align}
where the implied constant depends on $i,j$ but is independent of $a,b$.
\end{lemma}
\begin{proof}
Observe that $f(\x)$ is the composition of the function $h(\x):= t\log(x_1^2+x_2^2)$ with the linear map 
$$\x\mapsto \begin{pmatrix}  (a-b^2)^{1/2} & 0 \\b & 1\end{pmatrix}\x^T, $$
where $a-b^2> 0$ by the assumptions. Now one sees by a direct computation that 
\begin{align*} 
h^{(i,j)}(\x)= t\sum_{\substack{0\leq k\leq i, k\equiv i \, (2)\\ 0\leq l\leq j, l\equiv j \, (2)}} c_{k,l} \frac{x_1^k x_2^l}{(x_1^2+x_2^2)^{(i+j+k+l)/2}},\end{align*}
for some constants $c_{k,l}$. Thus we get the bound
\begin{align}\label{hbound} h^{(i,j)}(\x)\ll_{i,j} \frac{|t|}{(x_1^2+x_2^2)^{(i+j)/2}}, \end{align}
using the elementary inequality $xy\ll_\alpha x^{1/\alpha}+y^{1/(1-\alpha)} $ for $0<\alpha<1$.\\
By the chain rule we have
$$f^{(i,j)}(\x)= \sum_{l=0}^i \binom{i}{l}(a-b^2)^{(i-l)/2}b^{l} h^{(i-l,j+l)}((a-b^2)^{1/2} x_1,bx_1+x_2),  $$
and thus the results follows from (\ref{hbound}) since $b$ is bounded. 
\end{proof} 
From this we can now prove the following bounds.
\begin{lemma} \label{hessbound} Let $\m$, $\x$ and $\X$ satisfy the constraints coming from Lemma \ref{Wdiff}.Then we have
\begin{align}
\label{normbound}\left| g_{\m}^{(i,j)}(\x)\right| &\ll a^{i/2}\frac{|t| Q(\m )^{1/2}}{Q(\X)^{(i+j+1)/2}},\\ 
\label{detbound} \det (\Hess(g_{\m}(\x)) &\gg   \left(a^{1/2} \frac{|t| \, Q(\m)^{1/2}}{Q(\X)^{3/2}} \right)^2.
\end{align}
\end{lemma}
\begin{proof}
It follows from a two-dimensional Taylor expansion that 
\begin{align}  
\label{taylor} g_{\m}(\x) = \sum_{\alpha\in \{(1,0),(0,1)\}} \frac{1}{\alpha!}\, \m^\alpha \int_0^1 f^{\alpha}(\x+t\m) dt,              
\end{align}
where we use the multi-exponential notation $(x_1,x_2)^{(i,j)}:= x_1^ix_2^j$.\\
Using Lemma \ref{boundlemma}, we see that for $\alpha=(\alpha_1,\alpha_2)\in \{(1,0),(0,1)\}$, we have
\begin{align*} \m^\alpha f^{\alpha+(i,j)}(\x)  &\ll_{i,j} \m^{\alpha} \frac{|t|a^{(\alpha_1 +i)/2}}{Q(\X)^{(i+j+1)/2}}\\
&\ll  a^{i/2}\frac{|t| Q(\mu)^{1/2}}{Q(\X)^{(i+j+1)/2}},
\end{align*}
using that $\mu_1 a^{1/2}\ll Q(\m)^{1/2}$, respectively $\mu_2\ll Q(\m)^{1/2}$.\\
Thus by applying $\frac{\partial^{i+j}}{\partial x_1^i \partial x_2^j}$ term by term in (\ref{taylor}) and the bound above, we conclude (\ref{normbound}).\\

To prove the last inequality, we apply the following direct computation;
$$ \det (\Hess(g_{\m}(\x))= \frac{t^2(\det Q) Q(2\x+\m)Q(\m)}{Q(\x)^2Q(\x+\m)^2}\gg a \frac{|t|^2 Q(\m)}{Q(\X)^3}, $$
where we used $||\m||=o(\min (X_1,X_2))$.

\end{proof}

\subsection{Proof of Proposition \ref{mainprop}} 
Now we would like to apply Lemma \ref{sumint}, but obviously we need to alter $g_{\m}$ a bit in order for the conditions on the derivatives to be satisfied. We observe that the maximum variation in $\Delta_{p,q}$ of $  g^{(1,0)}_{\m} $ is bounded by
$$ l_1 \cdot \max_{\x \in \Delta_{p,q}} \left| g^{(2,0)}_{\m}(\x) \right| +  l_2 \cdot \max_{\x \in \Delta_{p,q}}  \left|  g^{(1,1)}_{\m}(\x) \right| \ll |t|^{-\eps}, $$
where we used (\ref{normbound}), and similarly for $ g^{(0,1)}_{\m}$, in which case the variation is even smaller.\\
Thus for sufficiently large $t$ the variation in each sub-sum $S_{p,q}$ is less than $\pi$, which was exactly why we chose $l_1,l_2$ as in (\ref{l}). Thus (following Titchmarsh) we can, associated to each $\Delta_{p,q}$, find integers $M,N$ such that 
$$ G_{\m}(\x):= g_{\m}(\x)-2\pi M x_1-2\pi N x_2,$$
satisfies 
$$\left| G^{(1,0)}_{\m}(\x)\right|\leq 3\pi/2 \quad \text{and}\quad    \left| G^{(0,1)}_{\m}(\x) \right| \leq 3\pi/2,$$
for all $\x\in\Delta_{p,q}$. Thus we get by Lemma \ref{sumint} 
\begin{equation}\label{sumtoint} \sum_{\x \in \Delta_{p,q}} e^{i g_{\m}(\x)}=\sum_{\x \in \Delta_{p,q}} e^{i G_{\m}(\x)} = \int_{\Delta_{p,q}} e^{iG_{\m}(\x)}d\x +O(l). \end{equation} 
Observe that all partial derivates of order at least two of $G_{\m}$ and $g_{\m}$ coincide.   \\

We would like to apply Lemma \ref{vanderC}, but we cannot do this directly since the required lower bounds on the order two derivatives do not hold in general. By considering different cases and doing an appropriate change of variable, we can however put us in a situation where we can apply Lemma \ref{vanderC}. Titchmarsh makes similar considerations in the proof of \cite[Lemma $\zeta$]{Ti34} and on \cite[p. 497]{Ti34}, but his argument gets simplified by the fact that $G_{\m}^{(2,0)}=-a G_{\m}^{(0,2)}$ when $b=0$ (which is {\it not} true for $b\neq 0$).\\ 
The idea to deal with the non-diagonal case is quite simply to consider two cases; if the partial derivative $G_{\m}^{(1,1)}$ is small then the lower bound on the Hesse-determinant forces the two other partial derivatives to be large. If on the other hand $G_{\m}^{(1,1)}$ is large then after a change of variable, we can force the new partial derivatives $(2,0)$ and $(0,2)$ to be large. This will allow us to prove the following key lemma.  

\begin{lemma} \label{lemmatomainprop}
With notation as above we have
$$ \int_{\Delta_{p,q}} e^{ig_{\m}(\x)}d\x=\int_{\Delta_{p,q}} e^{iG_{\m}(\x)}d\x \ll |t|^{-1+\eps} \frac{Q(\X)^{3/2}}{ a^{1/2} Q(\m)^{1/2}}. $$
\end{lemma} 

\begin{proof} Firstly we make a change of variables to the new variables $\y=(y_1,y_2)=(a^{1/4} x_1,a^{-1/4}x_2)$, under which the integral becomes
\begin{align}\label{integraldelta} \int_{\tilde{\Delta}_{p,q}} e^{i\tilde{G}_{\m}(\y)}d\y, \end{align}
where $\tilde{G}_{\m}(\y)=G_{\m} (a^{-1/4}y_1,a^{1/4}y_2)$ and the new rectangle $\tilde{\Delta}_{p,q}$ has side lengths $(a^{1/4}l_1)\times (a^{-1/4}l_2)$.\\
The reason for doing this change of variable is that by the bounds in Lemma \ref{hessbound} and the chain rule, it now follows that all order two partial derivates of $\tilde{G}_{\m}$ are bounded by 
\begin{align}\label{r}\ll |t| a^{1/2} Q(\m)^{1/2} Q(\X)^{-3/2}=:r.\end{align}
Let $\lambda_1,\lambda_2>0$ be constants independent of $a,b$ and $t$ (large enough) such that
\begin{align}
\label{normbound1}|\tilde{G}_{\m}^{\alpha}(\y)| &\leq \lambda_1 r,\\
\label{hessbound1}|\tilde{G}_{\m}^{(2,0)}(\y)\tilde{G}_{\m}^{(0,2)}(\y)-(\tilde{G}_{\m}^{(1,1)}(\y))^2| &\geq \lambda_2 r^2 ,
\end{align}
for $\alpha\in \{(2,0),(1,1),(0,2)\}$ and $\y \in\tilde{\Delta}_{p,q} $. We now split into different cases depending on the sizes of the order two partial derivatives.\\

{\bf Case 1:} Assume that $(\tilde{G}_{\m}^{(1,1)}(\y))^2 < \lambda_2 r^2/2$ for all $\y \in \tilde{\Delta}_{p,q}$. \\
Then it follows from (\ref{hessbound1}) that 
$$|\tilde{G}_{\m}^{(2,0)}(\y)\tilde{G}_{\m}^{(0,2)}(\y)|> \lambda_2 r^2/2.$$
Thus we conclude using the bound (\ref{normbound1}) above
$$ \lambda_2 r^2/2 < |\tilde{G}_{\m}^{(2,0)}(\y)\tilde{G}_{\m}^{(0,2)}(\y)| < \lambda_1 r |\tilde{G}_{\m}^{(2,0)}(\y)|, $$
and thus $|\tilde{G}_{\m}^{(2,0)}(\y)| \gg r$ and similarly for $\tilde{G}_{\m}^{(0,2)}(\y)$. The result now follows from Lemma \ref{vanderC}. \\

{\bf Case 2:} Assume that $|\tilde{G}_{\m}^{(1,1)}(\y)|^2 \geq \lambda_2 r^2/2$ for some  $\y\in \tilde{\Delta}_{p,q}$. \\ 
This we will show implies that for any $\delta>0$, we have 
 $$|\tilde{G}_{\m}^{(1,1)}(\y)| \geq (2^{-1/2}-\delta)\lambda_2^{1/2} r$$
for {\it all} $\y \in\tilde{\Delta}_{p,q} $ when $t$ is sufficiently large. To see this we bound the variation of $\tilde{G}_{\m}^{(1,1)}$ in $\tilde{\Delta}_{p,q}$; we observe that by the chain rule
$$\tilde{G}_{\m}^{(i,j)}(\y)= a^{(j-i)/4} G_{\m}^{(i,j)}(a^{-1/4}y_1,a^{1/4}y_2),$$
and thus by applying (\ref{normbound}), we can bound the variation of $\tilde{G}_{\m}^{(1,1)}$ in $\tilde{\Delta}_{p,q}$ by;
\begin{align} \nonumber &a^{1/4}l_1 \cdot \max_{\y\in \tilde{\Delta}_{p,q}} |\tilde{G}_{\m}^{(2,1)}(a^{-1/4}y_1,a^{1/4}y_2)|+a^{-1/4}l_2\cdot \max_{\y\in \tilde{\Delta}_{p,q}} |\tilde{G}_{\m}^{(1,2)}(a^{-1/4}y_1,a^{1/4}y_2)| \\
 \nonumber &=l_1 \cdot \max_{\x\in \Delta_{p,q}} |G_{\m}^{(2,1)}(\x)|+l_2\cdot \max_{\x\in \Delta_{p,q}} |G_{\m}^{(1,2)}(\x)| \\
\nonumber &\ll \frac{Q(\X)^{3/2}}{a Q(\m)^{1/2} |t|^{1+2\eps}}\cdot \frac{a Q(\m)^{1/2} |t|}{Q(\X)^2}+\frac{Q(\X)^{3/2}}{a^{1/2} Q(\m)^{1/2} |t|^{1+2\eps}}\cdot \frac{a^{1/2} Q(\m)^{1/2} |t|}{Q(\X)^2}\\
\nonumber&\ll |t|^{-2\eps} Q(\X)^{-1/2},
\end{align}
which is $o(r)$ as $t\rightarrow \infty$ since $Q(\X) \ll a^{1/2} |t|^{1+\eps}$ (recall the definition (\ref{r}) of $r$). Now we have two further sub-cases.\\ 

{\bf Case 2.1:} If 
$$|\tilde{G}_{\m}^{(2,0)}(\y)|,|\tilde{G}_{\m}^{(0,2)}(\y)|> 2^{-2} \lambda_1^{-1} \lambda_2 r, $$ 
for {\it all} $\y \in \tilde{\Delta}_{p,q}$, then we can apply Lemma \ref{vanderC} directly.\\ 
{\bf Case 2.2:} So we may assume that, say, $|\tilde{G}_{\m}^{(2,0)}(\y)|\leq   2^{-2}\lambda_1^{-1} \lambda_2 r  $ for some $\y \in \tilde{\Delta}_{p,q}$. As above, we see using (\ref{normbound}) that the variation of $\tilde{G}_{\m}^{(2,0)}$ in $\tilde{\Delta}_{p,q}$ is bounded by
\begin{align} \nonumber &a^{1/4}l_1 \cdot \max_{\y\in \tilde{\Delta}_{p,q}} |\tilde{G}_{\m}^{(3,0)}(a^{-1/4}y_1,a^{1/4}y_2)|+a^{-1/4}l_2\cdot \max_{\y\in \tilde{\Delta}_{p,q}} |\tilde{G}_{\m}^{(2,1)}(a^{-1/4}y_1,a^{1/4}y_2)| \\
 \nonumber &= a^{-1/2} l_1 \cdot \max_{\x\in \Delta_{p,q}} |G_{\m}^{(3,0)}(\x)|+a^{-1/2}l_2\cdot \max_{\x\in \Delta_{p,q}} |G_{\m}^{(2,1)}(\x)| \\
\nonumber &\ll \frac{Q(\X)^{3/2}}{a^{3/2} Q(\m)^{1/2} |t|^{1+2\eps}}\cdot \frac{a^{3/2} Q(\m)^{1/2} |t|}{Q(\X)^2}+\frac{Q(\X)^{3/2}}{aQ(\m)^{1/2} |t|^{1+2\eps}}\cdot \frac{a Q(\m)^{1/2} |t|}{Q(\X)^2}\\
\nonumber&\ll |t|^{-2\eps} Q(\X)^{-1/2},
\end{align}
which as above is $o(r)$ as $t\rightarrow \infty$. Thus we conclude that for any $\delta'>0$;
$$|\tilde{G}_{\m}^{(2,0)}(\y)|\leq   (2^{-2}+\delta')\lambda_1^{-1} \lambda_2 r  $$ 
holds for {\it all} $\y\in\tilde{\Delta}_{p,q}$ when $t$ is sufficiently large. \\
If we write 
\begin{align}\label{cov} \z=(z_1,z_2)=( dy_1-cy_2, dy_1+cy_2), \end{align}
with $cd=1/2$, then after a change of variable the integral (\ref{integraldelta}) becomes 
$$ \int_{\Omega_{p,q}}e^{ih(\z)}d\z, $$
where $h(\z)=\tilde{G}_{\m}(cz_1+cz_2, -dz_1+dz_2)$ and $\Omega_{p,q}$ is a new rectangle with angle $\pi/4$ to the coordinate axis and maximum side length $\ll a^{1/4}l_1 \max(c,d)$. We observe that 
\begin{align*}
h^{(2,0)}&= c^2\tilde{G}_{\m}^{(2,0)}+d^2\tilde{G}_{\m}^{(0,2)}-\tilde{G}_{\m}^{(1,1)},\\
h^{(0,2)}&= c^2\tilde{G}_{\m}^{(2,0)}+d^2\tilde{G}_{\m}^{(0,2)}+\tilde{G}_{\m}^{(1,1)},\\
h^{(1,1)} &= c^2\tilde{G}_{\m}^{(2,0)}-d^2\tilde{G}_{\m}^{(0,2)}.
\end{align*}
Thus by choosing $c=\lambda_1^{1/2}\lambda_2^{-1/4}, d=\lambda_1^{-1/2}\lambda_2^{1/4}/2$ and $\delta,\delta'$ sufficiently small, we get for {\it all} $\z\in \Omega_{p,q}$ the following bounds;
$$ r \ll (2^{-1/2}-1/2-\delta-\delta')\lambda_2^{1/2} r\leq |h^{(2,0)}(\z) |, |h^{(0,2)}(\z)| \ll r,\quad |h^{(1,1)}(\z)|\ll r. $$
Since the determinant of the Hesse-matrix is unchanged under the change of variable corresponding to (\ref{cov}), the result follows from Lemma \ref{vanderC}. Observe that the implied constant we get from Lemma \ref{vanderC} is indeed uniform in $a,b$ and $t$ since the angles of the rectangles $\Omega_{p,q}$ to the coordinate axes are fixed.  
\end{proof} 

We are now ready to finish the proof of our main theorem.
\begin{proof}[Proof of Proposition \ref{mainprop} and Theorem \ref{mainthm}] 
Combining (\ref{sumtoint}) and Lemma \ref{lemmatomainprop}, we get the following bound for all $\m$ as above;   
\begin{align*}
S'(\m)&= \sum_{p,q}S_{p,q}(\m)\\ 
&\ll  \sum_{p,q} \left(|t|^{-1+\eps} \frac{Q(\X)^{3/2}}{a^{1/2}Q(\m)^{1/2}}+l_2\right) \\
&\ll  \frac{X_1X_2}{a^{-3/2}Q(\X)^3 |t|^{-2-4\eps} Q(\m)^{-1}} \cdot \left( |t|^{-1+\eps} \frac{Q(\X)^{3/2}}{a^{1/2}Q(\m)^{1/2}}+\frac{Q(\X)^{3/2}}{a^{1/2} |t|^{1+2\eps}Q(\m)^{1/2}}\right)\\
&\ll a^{1/2} \frac{|t|^{1+5\eps} Q(\m)^{1/2}}{Q(\X)^{1/2}},
\end{align*} 
where we used $a^{1/2}X_1X_2\ll Q(\X)$. Plugging this into Lemma \ref{Wdiff} yields;
\begin{align*} \frac{1}{Q(\X)^{1/2}}\sum_{\substack{X_1 \leq x_1\leq X_1'\\ X_2\leq x_2\leq X_2'}} & e^{if(x_1,x_2)}\\
&\ll  \frac{X_1X_2}{Q(\X)^{1/2}\rho} +\frac{(X_1X_2)^{1/2}}{Q(\X)^{1/2}\rho}\left( \sum_{0\leq \mu_1, \mu_2\leq  \rho} \frac{a^{1/2} |t|^{1+5\eps} Q(\m)^{1/2}}{Q(\X)^{1/2}} \right)^{1/2}\\
&\ll \frac{Q(\X)^{1/2}}{a^{1/2}\rho}+ \frac{|t|^{1/2+3\eps}}{ Q(\X)^{1/4}\rho}\left( \sum_{0\leq \mu_1, \mu_2\leq  \rho} Q(\m)^{1/2} \right)^{1/2}\\
&\ll \frac{Q(\X)^{1/2}}{a^{1/2}\rho}+ \frac{|t|^{1/2+3\eps}a^{1/4}}{ Q(\X)^{1/4}\rho}\left( \sum_{||\m||\leq \rho} ||\m|| \right)^{1/2}\\
&\ll \frac{Q(\X)^{1/2}}{a^{1/2}\rho}+ \frac{|t|^{1/2+3\eps}a^{1/4} \rho^{1/2} }{ Q(\X)^{1/4}}.
 \end{align*}
Finally we choose an integer $ \rho \asymp Q(\X)^{1/2}|t|^{-1/3} a^{-1/2} $ to balance the terms, which yields the desired bound $\ll_\eps |t|^{1/3+3\eps}$. This choice of $\rho$ is admissible with respect to the conditions in Lemma \ref{Wdiff} since first of all
$$ \rho \ll a^{1/4} |t|^{1/2+\eps} |t|^{-1/3}a^{-1/2} =|t|^{1/6+\eps}a^{-1/4}, $$
which is less than $X_1$ and $X_2$ by (\ref{lowerbound}) and secondly we have $\rho \gg 1$, which again follows from (\ref{lowerbound}). \\
This proves Proposition \ref{mainprop} and consequently we conclude the proof of Theorem \ref{mainthm}. \end{proof}

\section{Lower bounds for the sup norm and a conjecture} 
As a concluding remark we will make some consideration on the best possible bound of the type (\ref{bound1}). First of all the appearance of $y^\delta$ in (\ref{bound1}) is necessary in the sense that for a fixed $t$, the Eisenstein series is unbounded because of the constant Fourier coefficient. We will now show that the lower bound $\delta\geq 1/2$ holds for any bound of the form (\ref{bound1}) and state a uniform version of the sup norm conjecture for Eisenstein series. \\

We have for $t$ fixed the following crude bound for the $K$-Bessel function \cite[p. 60]{Iw};
$$ K_{it}(y)\ll_t y^{-1/2}e^{-y},$$
as $y\rightarrow \infty$. Thus from the Fourier expansion of the Eisenstein series \cite[Theorem 3.4]{Iw};
\begin{align*}  E(z,s)&=y^s+\varphi(s)y^{1-s}+ 4\sqrt{y} \sum_{n\geq 1} \frac{K_{s-1/2}(2\pi y n)\tau_{s-1/2}(n)}{\Gamma(s) \zeta(2s) \pi^{-s}}\, \cos(2\pi x n), \end{align*}
we see that 
\begin{align}\label{trivialbound}E(z,1/2+it)=y^{1/2+it}+\varphi(1/2+it)y^{1/2-it}+O_t(e^{-\pi y}).  \end{align}
Now observe that for fixed $t\geq 1$, we can choose arbitrarily large $y$ such that 
$$1+\varphi(1/2+it)y^{-2it}=2,$$
using that $|\varphi(1/2+it)|=1$.\\
For such $y$, we thus have
$$E(z,1/2+it)=y^{1/2}(2y^{it}+o_t(1))\gg y^{1/2},$$
when $t$ is sufficiently large. Since we can let $y\rightarrow \infty$, we conclude that any bound of the form (\ref{bound1}) has to satisfy $\delta\geq 1/2$.\\

One might speculate that the following holds for any $\eps>0$;
\begin{equation}\label{conjecture} \text{\bf Conjecture: }\quad E(z, 1/2+it)\ll_\eps y^{1/2} (|t|+1)^\eps , \end{equation}
uniformly for $z\in \mathcal{F}$, the standard fundamental domain (\ref{F}) for $\Gamma_0(1)$. Note that this conjecture together with (\ref{squareL}) implies simultaneous Lindelöf in the $t$-aspect and on average in the $D$-aspect for the family of class group $L$-functions of imaginary quadratic fields. \\

\bibliographystyle{amsplain}
\providecommand{\bysame}{\leavevmode\hbox to3em{\hrulefill}\thinspace}
\providecommand{\MR}{\relax\ifhmode\unskip\space\fi MR }
\providecommand{\MRhref}[2]{%
  \href{http://www.ams.org/mathscinet-getitem?mr=#1}{#2}
}
\providecommand{\href}[2]{#2}

\end{document}